\documentclass[leqno,a4paper]{article}
\usepackage{amsmath,amsthm,amssymb}
\usepackage[utf8]{inputenc}
\usepackage[T1]{fontenc}

\usepackage{enumerate}
\usepackage{bbm}
\usepackage{todonotes}
\usepackage{mathrsfs}
\usepackage{mathabx}
\usepackage{hyperref}
\usepackage{nicefrac}
\usepackage{tikz}
\usetikzlibrary{matrix}

\usepackage{url}
\makeatletter
\g@addto@macro{\UrlBreaks}{\UrlOrds}
\makeatother

\usepackage[sort,numbers]{natbib}

\providecommand{\noopsort}[1]{} %potrzebne, gdy nazwiska z prefiksami!

\newtheorem{Th}{Theorem}[section]
\newtheorem{Prop}[Th]{Proposition}

\theoremstyle{definition}
\newtheorem{Remark}[Th]{Remark}

\newtheorem{Cor}[Th]{Corollary}

\newcommand{\beq}{\begin{equation}}
\newcommand{\eeq}{\end{equation}}

\def\scalar(#1,#2){(#1\mid#2)}

\newcommand{\un}{\underline}
\newcommand{\raz}{\mathbbm{1}}

\newcommand{\cf}{{\cal F}}

\newcommand{\ov}{\overline}

\newcommand{\Z}{{\mathbb{Z}}}
\newcommand{\N}{{\mathbb{N}}}

\newcommand{\vep}{\varepsilon}

\newcommand{\mob}{\boldsymbol{\mu}}
\newcommand{\lio}{\boldsymbol{\lambda}}

\newcommand{\Emp}{\mathcal{E}}
\newcommand{\logEmp}{\Emp^{\rm log}}

\title{Sarnak's conjecture implies the Chowla conjecture along a subsequence}
\author{A.\ Gomilko\thanks{Research supported by Narodowe Centrum Nauki grant DEC-2014/13/B/ST1/03153}\\
Faculty of Mathematics and Computer Science\\
Nicolaus Copernicus University\\
Toru\'n, Poland\\\url{gomilko@mat.umk.pl}  \and D.\ Kwietniak\thanks{Research supported by Narodowe Centrum Nauki grant  UMO-2012/07/E/ST1/00185.}\\
Faculty of Mathematics and Computer Science\\
Jagiellonian University in Krak\'ow\\
Poland\\
and\\
Institute of Mathematics\\
Federal University of Rio de Janeiro\\
Brazil\\
\url{dominik.kwietniak@uj.edu.pl}
 \and M.\ Lema\'nczyk\thanks{Research supported by Narodowe Centrum Nauki grant  UMO-2014/15/B/ST1/03736.}\\
 Faculty of Mathematics and Computer Science\\
Nicolaus Copernicus University\\
Toru\'n, Poland\\\url{mlem@mat.umk.pl}
}

\date{\today}

\begin{document}

\maketitle

\begin{abstract} We show that the M\"obius disjointess of zero entropy dynamical systems implies the existence of an increasing sequence  of positive integers along which the Chowla conjecture on autocorrelations of the M\"obius function holds.
\end{abstract}

\section{Introduction} For the definitions, basic notation and results concerning the conjecture of Sarnak and the one of Chowla we refer the reader to the survey \cite{Fe-Ku-Le}. We only recall here that the Chowla conjecture implies Sarnak's conjecture \cite{Ab-Ku-Le-Ru}, \cite{Sa}. Then the intriguing question arises whether the reverse implication is true. Motivated by some recent results concerning logarithmic autocorrelations of the classical M\"obius  function $\mob\colon\N\to\{-1,0,1\}$ (and the Liouville function $\lio\colon\N\to\{-1,1\}$), see \cite{Fr}, \cite{Fr-Ho1}, \cite{Ta3}, \cite{Ta} and using Tao's result \cite{Ta} on the equivalence of logarithmic versions of Sarnak's and Chowla conjectures, we give a partial answer to the aforementioned question by showing that:

\begin{Th}\label{t:sartocho}
If Sarnak's conjecture is satisfied then there exists an increasing sequence $(N_k)$ of positive integers along which the Chowla conjecture holds.
\end{Th}

As a matter of fact, we show that the assertion of Theorem~\ref{t:sartocho} follows from the logarithmic version of the Chowla conjecture.

\section{Ces\`aro and harmonic limits of empirical measures} \subsection{Ergodic components of members of $V(x)$ and $V^{\rm log}(x)$}
Let $X$ be a compact metric space. By $M(X)$ we denote the space of probability Borel measures on $X$, in fact, we will also consider the space of $\widetilde{M}(X)$ of Borel measures $\mu$ on $X$
for which $\mu(X)\leq 2$. With the weak-$\ast$-topology, $M(X)$ is a compact metrizable space and the metric we will consider is given by
\beq\label{ee1}
d(\mu,\nu)=\sum_{j\geq1}\frac1{2^j}\left|\int f_j\,d\mu-\int f_j\,d\nu\right|,
\eeq
where $\{f_j:\:j\geq1\}$ is a linearly dense set of continuous function whose sup norm is $\leq1$. Note that, by~\eqref{ee1}, if $0\leq\alpha\leq1$ then
$d(\alpha\mu,\alpha\nu)=\alpha d(\mu,\nu)$. In particular, $d$ is convex:
\beq\label{ee2}
d\left(\int\mu_\gamma\,dP(\gamma) ,\int \nu_\gamma\,dP(\gamma)\right)\leq \int d(\mu_\gamma,\nu_\gamma)\,dP(\gamma),\eeq
where $P$ is a Borel probability measure on the set of indices $\gamma$ in some Polish metric space, $\gamma\mapsto\mu_\gamma$ is an $M(X)$-valued Borel function, and $\int\mu_\gamma\,dP(\gamma)$ denotes the Pettis integral.

Let $(X,T)$ be a dynamical system given by a continuous map $T\colon X\to X$ and $M(X,T)$ (respectively, $M^e(X,T)$) stands for the set of $T$-invariant (respectively, ergodic) measures. Given $x\in X$ and $n\in\mathbb{N}$, we write $\delta_{T^n(x)}$ for the Dirac measure concentrated at the point $T^n(x)$. Let $\Delta(x,N)$ denote the counting measure concentrated on $\{x,T(x),\ldots,T^{N-1}(x)\}$, where $N\ge 1$ and
let the \emph{empirical measure} $\Emp(x,N)$  be the normalized counting measure, that is,
\[
\Delta(x,N)=\sum_{n=1}^N \delta_{T^{n-1}(x)}, \qquad\text{and}\qquad \Emp(x,N)=\frac1N\Delta(x,N).
\]
We say that $x\in X$ is quasi-generic for a measure $\nu\in M(X)$ if for some subsequence $(N_k)$ we have $\Emp(x,N_k)\to\nu$.
The Ces\`aro limit set of $x$ is
$$
V(x):=\{\nu\in M(X):\:x\text{ is quasi-generic for }\nu\}.$$
The set was studied by many authors and it is well-known (and easy to see) that we always have
\beq\label{e1}
V(x)\subset M(X,T)\eeq
and $V(x)$ is a nonempty, closed and connected set, see \cite{De-Gr-Si}, hence
\beq\label{e2}
\mbox{either $|V(x)|=1$ or $V(x)$ is uncountable.}\eeq
Choosing different normalization method for the counting measures $\Delta(x,N)$, we arrive at the notion of harmonic limit set of a point.
Let
$$\logEmp(x,N)=\frac1{\log N}\sum_{n=1}^N\frac1n\delta_{T^{n-1}(x)}, \qquad\text{for }N\ge 2.
$$
Note that $\logEmp(x,N)$ is not a probability measure. In order to stay in $M(X)$, we should consider
$$
\Emp^{\rm log}_{\rm nrm}(x,N)=\frac1{H_N}\sum_{n=1 }^N\frac1n \delta_{T^{n-1}(x),}\text{ where }H_N=\sum_{n=1}^N\frac1n.$$
Note that the limit sets of sequences $(\logEmp(x,N))$ and $(\Emp^{\rm log}_{\rm nrm}(x,N))$ as $N$ goes to $\infty$ coincide, so by abuse of notation  we will not distinguish between $M(X)$ and $\widetilde{M}(X)$ and we will often deal with sequences of linear combinations of measures which are not exactly convex (affine), but are closer and closer to be so when we pass with the index $N$ to infinity. We say that $x\in X$ is logarithmically quasi-generic for a measure $\nu$ if for some subsequence $(N_k)$ we have
$\logEmp(x,N_k)\to\nu$ as $k\to\infty$.
The harmonic limit set of $x$ is defined as
$$
V^{\rm log}(x):=\{\nu\in M(X):\:x\text{ is logarithmically quasi-generic for }\nu\}.$$

%Since $$
%\left|\frac1{\log N}\sum_{n\leq N}\frac{\delta_{T^nx}}n-
%\frac1{\log N}\sum_{n\leq N}\frac{\delta_{T^{n+1}x}}n\right|=
%\frac1{\log N}\sum_{n\leq N}\frac{\delta_{T^nx}}{n(n+1)}+{\rm o}(1),$$
It is easy to see that the proofs of \eqref{e1}--\eqref{e2} presented for $V(x)$ in \cite{De-Gr-Si} can be easily adapted to harmonic averages, so
we have that $V^{\rm log}(x)$ is nonempty, closed, connected, and consists of $T$-invariant measures. In particular, \eqref{e2} also holds for $V^{\rm log}(x)$. If $V(x)=\{\nu\}$, then $V^{\rm log}(x)=\{\nu\}$, but the converse need not be true. Nevertheless, the measures in the harmonic limit set of $x$ are always members of the closed convex hull of $V(x)$ (note that the latter set need not be convex as in a dynamical system with the specification property every nonempty compact and connected set $V\subset M(X,T)$ is the Ces\`aro limit set of some point $x\in X$, cf.\ \cite[Proposition 21.14]{De-Gr-Si}).

\begin{Prop}\label{p:col1}  We have
$V^{\rm log}(x)\subset\ov{{\rm conv}}(V(x))$.\end{Prop}
\begin{proof}
Fix $x\in X$ and let
$$
A:=\{\Emp(x,N):\:N\geq1\},\;\;B:=\{\logEmp(x,N):\:N\geq2\}.$$
%where $a_n=\frac1n\sum_{i\leq n}\delta_{T^ix}$, $n\geq1$ and $b_n=\frac1{\log n}\sum_{i\leq n}\frac{\delta_{T^ix}}i$ for $n\geq 2$.
If $x$ is eventually periodic, that is, $T^k(x)=T^\ell(x)$ for some $0\leq k<\ell$, then it is easy to see that $V(x)=V^{\rm log}(x)=\{\Emp(T^k(x),\ell-k)\}$.
We will assume that $x$ is not an eventually periodic point. It follows that $A\cap M(X,T)=B\cap M(X,T)=\emptyset$. Furthermore, $\ov{A}=V(x)\cup A$ and $\ov{B}=V^{\rm log}(x)\cup B$, where the summands are disjoint in both cases. Note that
\[
\Delta(x,n)-\Delta(x,n-1)=\delta_{T^{n-1}(x)} \qquad\text{ for }n=2,3,\ldots.
\]
Using the summation by parts trick, we obtain
\begin{multline}\label{eq:conv}
\sum_{n=1}^N\frac1n\delta_{T^{n-1}(x)}=\Delta(x,1)+\sum_{n=2}^N\frac1n(\Delta(x,n)-\Delta(x,n-1))=\\
%\Delta(x,1)+\frac1N\Delta(x,N)-\Delta(x,1)+\sum_{n=1}^{N-1}\Delta(x,n)\left(\frac1n-\frac1{n+1}\right)=\\
\frac1N\Delta(x,N)+\sum_{n=1}^{N-1}\left(\frac1{(n+1)n}\right)\Delta(x,n)=
\Emp(x,N)+\sum_{n=1}^{N-1}\frac{\Emp(x,n)}{n+1}.
\end{multline}

%or in $B$ is $T$-invariant. Moreover, both sets $A$ and $B$ are discrete (in the relative topology), that is no $a_{n_0}$ is an accumulation point of %$A\setminus\{a_{n_0}\}$. We have
%\beq\label{ee3}\mbox{$V(x)=\ov{A}\setminus A$, $W(x)=\ov{B}\setminus B$ are non-empty closed subsets.}\eeq
%Moreover,
%\beq\label{ee4}
%(\ov{A}\setminus A)\cap M(X,T)=\ov{A}\setminus A,\; (\ov{B}\setminus B)\cap M(X,T)=\ov{B}\setminus B.\eeq

Fix $\vep>0$. Then, there exists $K=K_\vep\geq1$ such that
\beq\label{ee5}d(\Emp(x,n),V(x))<\vep\quad\text{for } n\geq K_\vep.\eeq
%If not, we can find a subsequence $a_{\ell_n}$ with $d(a_{\ell_n},\ov{A}\setminus A)\geq \vep_0$ and choosing a further subsequence if necessary, we can assume $a_{\ell_n}\to\nu\in\ov{A}$, whence by~\eqref{ee4}, $\nu\in A$ and is $T$-invariant, a contradiction.}

%Now, the classical summation by parts\footnote{That is, we use
%$$\frac1{\log N}\sum_{n\leq N}\frac{c_n}n=\frac1{\log N}\sum_{n\leq N}(C_{n+1}-C_n)\frac1n$$$$
%=\frac1{\log N}\sum_{n\leq N}C_n
%\left(\frac1n-\frac1{n+1}\right)+{\rm o}(1)=\frac1{\log N}\sum_{n\leq N}\frac{C_n}n\cdot\frac1{n+1}+{\rm o}(1),$$
%where $C_n:=\sum_{j\leq n}c_j$.} tells us that
%\beq\label{ee6}
%d\left( \frac1{\log N}\sum_{n\leq N}\frac{\delta_{T^nx}}n,   \frac1{\log N}\sum_{n\leq N}\left(\frac1n\sum_{j\leq n}\delta_{T^jx}\right)\cdot\frac1{n+1}\right)\to 0\eeq
%when $N\to\infty$. Then
Using \eqref{eq:conv}, we get
%\begin{multline}\label{}
\beq\label{eq:sum}
\Emp^{\rm log}_{\rm nrm}(x,N)=\frac1{H_N}\sum_{n=1}^K\frac{\Emp(x,n)}{n+1}+\frac{\Emp(x,N)}{H_N}+\frac1{H_N}\sum_{n=K+1}^{N-1}\frac{\Emp(x,n)}{n+1}
 \eeq
%\Emp^{\rm log}_{\rm nrm}(x,N)=\frac{H_K-1}{H_N}\frac1{H_K-1}\sum_{n=1}^K\frac{\Emp(x,n)}{n+1}+
%\\\frac{H_N-H_K}{H_N}\frac{1}{H_N-H_K}\sum_{n=K+1}^{N-1}\frac{\Emp(x,n)}{n+1}
%+\frac{\Emp(x,N)}{H_N} \\
%
%\end{multline}
%\left(\frac1n\sum_{j\leq n}\delta_{T^jx}\right)\cdot\frac1{n+1}+$$$$
%\frac{\log(N-K)}{\log N}\frac1{\log(N-K)}\sum_{K<n\leq N}\left(\frac1n\sum_{j\leq n}\delta_{T^jx}\right)\cdot\frac1{n+1}.$$
Now, keeping $K$ fixed, we can assure that the total mass of the first two summands on the RHS of \eqref{eq:sum} is as close to~0 as we want provided that $N$ is large enough. Therefore, for every $N$ large enough, the measure
$\Emp^{\rm log}_{\rm nrm}(x,N)$ is $\vep$-close to
$$
\xi_N=\frac1{H_N-H_K}\sum_{n=K+1}^{N-1}\frac{\Emp(x,n)}{n+1}.
$$
The latter measure is an affine combination of $\Emp(x,n)$ for $n>K_\vep$ which are all $\vep$-close to $V(x)$. Using \eqref{ee2},  we get that $\xi_N$ is $\vep$-close to $\ov{\rm conv}(V(x))$, thus $d(\xi_N,\Emp^{\rm log}_{\rm nrm}(x,N))<2\vep$ for all $N$ large enough.
%$$
%d\left( \frac1{\log(N-K)}\sum_{K<n\leq N}\left(\frac1n\sum_{j\leq n}\delta_{T^jx}\right)\cdot\frac1{n+1},\ov{\rm conv}(\ov{A}\setminus A)\right)<2\vep$$
%if $N$ is large enough.
Putting all this together, we obtain that for each $\rho\in V^{\rm log}(x)$, we have
$$
d(\rho,\ov{\rm conv}(V(x)))=0$$
which completes the proof of Proposition~\ref{p:col1}.\end{proof}
%Suppose now that there exists a Borel set $R\subset M^e(X,T)$ such %that
%$$
%V(x)=\ov{A}\setminus A\subset \ov{\rm conv}(R).$$
%Then
%$$
%\ov{\rm conv}(R)=\{\int_R\mu\,dP(\mu):\: P\in M(R)\}\footnote{The %RHS is convex and since $M(R)
%However, if we consider a generalized convex combination $\int \nu_\gamma\, dP(\gamma)$ and we want to see its ergodic decomposition then it is simply first to take ergodic decomposition of each $\nu_\gamma$ and put all of them together. This yields the proof of Corollary~\ref{c:col1}.

%\subsection{Examples}
%Construct  examples:

%a) $|W(x)|=1$ and $V(x)$ uncountable;

%b)  $W(x)\setminus V(x)$ uncountable;

%c)  the set of ergodic components of $W(x)$ uncountable and the set of ergodic components of $V(x)$ is countable  - IN CONTRADICTION with Proposition~1.4!!

\subsection{Ergodic measures in $V^{\rm log}(x)$}

It turns out that if $x\in X$ is logarithmically quasi-generic for an ergodic measure then $x$ is also quasi-generic in the classical (Ces\`aro) sense.
%for this measure.

\begin{Cor}\label{c:milman} If an ergodic measure $\nu\in V^{\rm log}(x)$, then $\nu\in V(x)$.\end{Cor}
\begin{proof}
We first recall Milman's theorem (\cite{Ru}, Chapter 1.3, Theorem 3.25): {\em If $K$ is a compact set in a locally convex space and $\ov{\rm conv}(K)$ is compact then ${\rm ex}(\ov{\rm conv}(K))\subset K$.} Here, and elsewhere by ${\rm ex}(L)$ we denote the set of extreme points of a convex set $L$.
We apply Milman's result to $K=V(x)$ and ${\rm ex}(M(X,T))=M^e(X,T)$. By Proposition~\ref{p:col1},
%Proposition~\ref{p:col1}
we have
$$
V^{\rm log}(x)\cap M^e(X,T)\subset \ov{\rm conv}(V(x))\cap M^e(X,T)\subset
{\rm ex}(\ov{\rm conv}(V(x)))\subset V(x).\qedhere$$ \end{proof}

%In other words:

\subsection{Ergodic components of measures in $V^{\rm log}(x)$}
Let us recall that in our setting $M(X,T)$ is a metrizable, compact and convex subset of a locally convex space. It follows that we can use Choquet's representation theorem (see \cite{Ph}, Chapters 3 and 10) to conclude that if $L$ is a closed convex subset of $M(X,T)$ and $\kappa\in L$ then there exists a Borel probability measure $P_\kappa$ on $M(X,T)$ supported by ${\rm ex}(L)$ such that
\beq\label{eq:barycenter}
\kappa=\int_{{\rm ex}(L)}\rho\,dP_\kappa(\rho).\eeq
Furthermore, if $L=M(X,T)$, then $P_\kappa$ satisfying \eqref{eq:barycenter} is unique and the map $M(X,T)\ni \kappa\mapsto P_\kappa\in M(M^e(X,T))$ is Borel measurable \cite{Do} Fact A.2.12 and \cite{Ph} Prop.\ 11.1.

\begin{Prop}\label{p:logergcomp}
The set of ergodic components of measures in $V^{\rm log}(x)$ is contained in the set of ergodic components of measures in $V(x)$. More precisely: if a Borel set $\mathcal{D}\subset M^e(X,T)$ satisfies
$$
P_\mu(\mathcal{D})=1\text{ for each }\mu\in V(x)$$
then
$$
P_\kappa(\mathcal{D})=1\text{ for each }\kappa\in V^{\rm log}(x).$$

In particular, if the set of ergodic components of measures in $V(x)$ is countable, so is the set of ergodic components of $V^{\rm log}(x)$.\end{Prop}
\begin{proof}
Fix $\kappa\in V^{\rm log}(x)$. It follows from Proposition~\ref{p:col1} that $\kappa\in\ov{\rm conv}(V(x))$.
By Choquet's theorem, there exists $Q\in M(M(X,T))$ with $Q({\rm ex}(\ov{\rm conv}(V(x)))=1$ such that
$$
\kappa=\int_{{\rm ex}(\ov{\rm conv}(V(x)))}\mu\,dQ(\mu).$$
Using Milman's theorem, we obtain that $Q(V(x))=1$, so that
$$
\kappa=\int_{V(x)}\mu\,dQ(\mu).$$
Hence
$$
\kappa=\int_{V(x)}\left(\int_{M^e(X,T)}\rho\,dP_\mu(\rho)\right)\,dQ(\mu)=
\int_{M^e(X,T)}\rho\,dR(\rho),
$$
where $R\in M(M(X,T))$ is defined by
$$
R(\mathcal{C}):=\int_{V(x)}P_\mu(\mathcal{C})\,dQ(\mu)\text{ for a Borel subset }\mathcal{C}\subset M(X,T).$$
Note that the definition of $R$ is correct as $V(x)$ is Borel and $\mu\mapsto P_\mu$ is Borel measurable. Now, since $M(X,T)$ is a simplex, there is only one measure on $M^e(X,T)$ satisfying \eqref{eq:barycenter} and we obtain that $R=P_\kappa$. Since $P_{\mu}(\mathcal{D})=1$ for each $\mu\in V(x)$, we have $P_{\kappa}(\mathcal{D})=1$.\end{proof}

\section{Relations between Sarnak's conjecture and the Chowla conjecture}
Thinking of the M\"obius function as of a point $\mob$ in the sequence space $\{-1,0,1\}^\N$, we can consider the M\"obius dynamical system $(X_{\mob},S)$, where $X_{\mob}$ stands for the orbit closure of $\mob$ and $S$ denotes the left shift. %by $\mob\in\{-1,0,1\}^\Z$, where $\mob_n=\mu(|k|)$ for $k\neq 0$, $\mob_0=0$,
We now apply results from the previous sections to the M\"obius system and sets $V(\mob)$ and $V^{\rm log}(\mob)$. % $\mob$ itself as a quasi-generic point (for the logarithmic and Ces\`aro averages).

\vspace{2ex}

%\noindent
\begin{proof}[Proof of Theorem~\ref{t:sartocho}.]
It is obvious that Sarnak's conjecture implies logarithmic Sarnak's conjecture which, by a result of Tao \cite{Ta}, implies
the logarithmic version of Chowla conjecture.
The logarithmic version of Chowla conjecture for $\mob$ phrased in the language of ergodic theory means that
$$
\frac1{\log N}\sum_{n=1}^N\frac1n\delta_{S^{n-1}(\mob)}\to \widehat{\nu}_{\mob^2}, \quad \text{as }N\to\infty,$$
where $\widehat{\nu}_{\mob^2}$ is the relatively independent extension of the Mirsky measure $\nu_{\mob^2}$ of the square-free system $(X_{\mob^2},S)$ \cite{Ab-Ku-Le-Ru}. In particular, $\widehat{\nu}_{\mob^2}$ is ergodic. Equivalently, the conjecture says that $\widehat{\nu}_{\mob^2}\in V^{\rm log}(\mob)$. It follows from Corollary~\ref{c:milman} that $\mob$ is a quasi-generic point (in the sense of Ces\`aro) for $\widehat{\nu}_{\mob^2}$, that is, there is a sequence $(N_i)$ such that for each $1\leq a_1<\ldots<a_k$ and each choice of $j_0,j_1,\ldots,j_k\in\{1,2\}$ not all equal to~2, we have
$$
\lim_{i\to\infty}\frac1{N_i}\sum_{n\leq N_i}\mob^{j_0}(n)\mob^{j_1}(n+a_1)\cdot\ldots\cdot\mob^{j_k}(n+a_k)=0,$$
i.e., we obtain the Chowla conjecture along the subsequence $(N_i)$.\end{proof}

\begin{Remark}
If instead of the M\"obius function $\mob$ we consider the Liouville function $\lio$ then the Chowla conjecture claims that the limit is the Bernoulli measure $B(1/2,1/2)$ on $\{-1,1\}^{\N}$.
\end{Remark}

\vspace{2ex}

 In fact, using \cite{Fr}, we have the following:
\begin{Cor}\label{c:Fran}
Assume that there exists an ergodic measure $\kappa\in V^{\rm log}(\mob)$. Then there exists an increasing sequence $(N_i)$ such  that the Chowla conjecture holds along $(N_i)$.\end{Cor}
\begin{proof} If there exists $\kappa\in V^{\rm log}(\mob)\cap M^e(X_{\mob},S)$, then, reasoning as above, we see that for a subsequence $(N_i)$, we have $$\frac1{N_i}\sum_{1\leq n\leq N_i}\delta_{S^{n-1}(\mob)}\to \kappa\qquad\text{as }i\to\infty.$$ Now, by \cite{Fr}, we get $\kappa=\widehat{\nu}_{\mob^2}$. The result follows.\end{proof}

\begin{Remark} Assume that $(X,T)$ is a dynamical system and $x\in X$  is completely deterministic (i.e.\ each member $\kappa\in V(x)$ yields zero entropy measurable system $(X,\kappa,T)$) such that the ergodic components of all measures from $V(x)$ give a countable set. It follows from Proposition~\ref{p:logergcomp} and \cite{Fr-Ho1} (see Remarks after Theorem~1.3 in \cite{Fr-Ho1}) that, at $x$, we obtain M\"obius disjointness in the logarithmic sense.\end{Remark}

\section{Examples}
We collect here a couple of examples demonstrating that some of our results are optimal and cannot be improved.
Our examples are points in the full shift $\{0,1\}^\N$ or $\{0,1,2\}^\N$ constructed so that the Ces\`aro and harmonic limit sets are easy to identify. We will routinely omit some easy computations used in our proofs.

Let $\un{d}(J)$, $\ov{d}(J)$, $\un{\delta}(J)$, and $\ov{\delta}(J)$ denote, respectively the lower/upper asymptotic and lower/upper logarithmic density of a set $J\subset \Z$. It is well known that we always have $\un{d}(J)\le\un{\delta}(J)\le\ov{\delta}(J)\le\ov{d}(J)$. We write $\delta(J)$ for the common value of $\un{\delta}(J)$ and $\ov{\delta}(J)$ (if such an equality holds).

Our approach is based on the following criterion for logarithmic genericity. It can be proved the same way as for Ces\`aro averages the only difference is that logarithmic (harmonic) averages replace asymptotic averages.
\begin{Prop}
A point $x=(x_n)\in \{0,1\}^\N$ is logarithmically generic for a measure $\mu$ if and only if for every $k \ge 1$ and for every finite block
$w\in\{0,1\}^k$, the set of positions $j$ such that $w$ appears at the position $j$ in $w$, that is
the set $J_w=\{j\in\N: x_{j}=w_1,\ldots, x_{j+k-1}=w_k\}$
satisfies  $\un{\delta}(J_w)=\ov{\delta}(J_w)=\mu(\{y:y_{[0,k)}=w\})$.
\end{Prop}

We will also apply the following observation.

\begin{Prop}
For every point $x=(x_n)\in \{0,1\}^\N$,  $k \ge 1$ and  finite block
$w\in\{0,1\}^k$, we have
$$\un{d}(J_w)=\min\{\nu(\{y:y_{[0,k)}=w\}):\:\nu\in V(x)\}$$
and
$$\ov{d}(J_w)=\max\{\nu(\{y:y_{[0,k)}=w\}):\:\nu\in V(x)\},$$
where $J_w=\{j\in\N: x_{j}=w_1,\ldots, x_{j+k-1}=w_k\}$.
\end{Prop}
Both results are well-known. We now present our examples.

\begin{Prop} It can happen that $V^{\rm log}(x)\subsetneq V(x)$.
\end{Prop}
\begin{proof}
Let
$$x=0 11 0000 1111 11111 0000 0000 0000 0000 \ldots,$$
that is, %$x_0=0$ and
$$
x_n=\begin{cases}
      0, & \mbox{if }2^{2k}\le n+1<2^{2k+1}\text{ for some }k\ge 0,  \\
      1, & \mbox{if }2^{2k-1}\le n+1<2^{2k}\text{ for some }k\ge 1.
    \end{cases}
$$
It is easy to see that
$V(x)=\{\alpha\delta_{\ov{0}}+(1-\alpha)\delta_{\ov{1}}:1/3\le\alpha\le 2/3\}$, where $\delta_{\ov{p}}$ denotes the Dirac measure concentrated at the fixed point $ppp\ldots\in\{0,1\}^\N$. We claim that $V^{\rm log}(x)=\{1/2\delta_{\ov{0}}+1/2\delta_{\ov{1}}\}$. Indeed, it is easy to see that:
\begin{enumerate}
\item $\un{\delta}(J_w)=\ov{\delta}(J_w)=1/2$ for $w=0$ and $w=1$,
\item $\un{d}(J_w)=\ov{d}(J_w)=0$ for any block $w$ containing $01$ or $10$ as a subblock,
\item $\un{\delta}(J_w)=\ov{\delta}(J_w)=1/2$ for any $k\ge 2$ and $w=0^k$ or $w=1^k$.
\end{enumerate}
Our claim is an immediate consequence of these three observations. The proofs of the first two are based on easy computations. To see the third one, fix $k\ge 2$ and consider $w=1^k$ (the case $w=0^k$ is proved in the same way). Note that $J_1\setminus J_w$ can be equivalently described as the set of positions $j$ such that the block $x_jx_{j+1}\ldots x_{j+k-1}$ starts with $1$ and contains $10$ as a subblock. By 2.~this set satisfies $\ov{d}(J_1\setminus J_w)=0$, thus $\un{\delta}(J_w)=\un{\delta}(J_1)$ and $\ov{\delta}(J_w)=\ov{\delta}(J_1)$. It follows from~1.\ that $\delta(J_w)=1/2$.\end{proof}

We are grateful to J.\ Ku\l aga-Przymus for the following remark. It has also inspired our next proposition presenting a simpler example of the same phenomenon.
\begin{Remark}[J.\ Ku\l aga-Przymus] It is implicit in \cite{Ba-Ka-Ku-Le} that for each $\mathscr{B}\subset\N\setminus\{1\}$ which is not Besicovitch, for the subshift $(X_{\eta},S)$, where $\eta:=\raz_{\cf_{\mathscr{B}}}$ ($\cf_{\mathscr{B}}$ stands for the set of $\mathscr{B}$-free numbers), we have $V^{\rm log}(\eta)=\{\nu_{\eta}\}$ and $\nu_\eta$ (so called Mirsky measure of $(X_\eta,S)$) is ergodic, while $V(\eta)$ is uncountable, whence the set of ergodic components of members in $V(\eta)$ is strictly larger than the analogous set for $V^{\rm log}(\eta)$.
\end{Remark}
\begin{Prop} The set of ergodic measures appearing in the ergodic decompositions of members of $V(x)$ can be strictly bigger than the set of ergodic components of analogous set for $V^{\rm log}(x)$.
\end{Prop}
\begin{proof}
Let $x=(x_n)\in\{0,1\}^\N$, where
$$
x_n=\begin{cases}
      1, & \mbox{if }2^{k^2-1}\le n+1<2^{k^2},\text{ for some }k\ge 1,  \\
      0, & \mbox{\text{otherwise}}.
    \end{cases}
$$
Then it follows either from \cite[Lemma 2]{Lu-Po} or from direct computations that
$\ov{\delta}(J_1)=\ov{\delta}(\{0\le j<2^{k^2}:x_j=1\})=0$, which implies that $V^{\rm log}(x)=\{\delta_{\ov{0}}\}$. On the other hand for every $k\ge 1$ we have
\[
|\{0\le j<2^{k^2}:x_j=1\}|=\sum_{j=1}^k 2^{j^2-1}.
\]
Therefore $d(\{0\le j<2^{k^2}:x_j=1\})=1/2$ and we conclude that $V(x)\neq\{\delta_{\ov{0}}\}$.
\end{proof}

\begin{Prop} The sets $V(x)$ and $V^{\rm log}(x)$ can be disjoint.
\end{Prop}
\begin{proof}
Let
$$x=00 111111 222222222222222220000000000000\ldots,$$
that is, %$x_0=0$ and
$$
x_n=\begin{cases}
      0, & \mbox{if }3^{3k}\le n+1<3^{3k+1}\text{ for some }k\ge 0,  \\
      1, & \mbox{if }3^{3k+1}\le n+1<3^{3k+2}\text{ for some }k\ge 0, \\
      2, & \mbox{if }3^{3k+2}\le n+1<3^{3k+3}\text{ for some }k\ge 0.
    \end{cases}
$$
We claim that $V^{\rm log}(x)=\{1/3\delta_{\ov{0}}+1/3\delta_{\ov{1}}+1/3\delta_{\ov{2}}\}$ and $V^{\rm log} (x)\cap V(x)=\emptyset$.
We proceed as in the first example. First, we note that
the asymptotic density of the set $J_w$ of appearances of a block $w=w_1w_2\in\{0,1,2\}^2$ with $w_1\neq w_2$ in $x$ is zero. It follows that the support of every measure $\nu\in V(x)\cup V^{\rm log}(x)$ is contained in the set $\{\ov{0},\ov{1},\ov{2}\}$ consisting of three shift-invariant (fixed) points.
Next we compute (or conclude from \cite[Lemma 2]{Lu-Po} ) that
\[
\delta(\{j\in\N:x_j=0\})=\delta(\{j\in\N:x_j=1\})=\delta(\{j\in\N:x_j=2\})=1/3.
\]
This shows that $V^{\rm log}(x)=\{1/3\delta_{\ov{0}}+1/3\delta_{\ov{1}}+1/3\delta_{\ov{2}}\}$.

Let $\nu\in V(x)$. It follows that there is a sequence $(N_i)$ such that for every $k\ge 1$ and every block  $w\in\{0,1,2\}^k$, we have
\[
\lim_{i\to\infty}\frac{|\{0\le j< N_i: x_{j}=w_1,\ldots, x_{j+k-1}=w_k\}|}{N_i} =
\nu(\{y:y_{[0,k)}=w\}).
\]
Without loss of generality we can assume that there exists $p\in\{0,1,2\}$ such that $N_i \equiv p \bmod 3$ for every $i\in\N$. Let $q= p-1\bmod 3$.
It is then easy to see that
\[
 \nu(\{y:y_0=q\})\le \frac{3}{13}<\frac13,
\]
which implies that $\nu\neq 1/3\delta_{\ov{0}}+1/3\delta_{\ov{1}}+1/3\delta_{\ov{2}}$.
%On the other hand, for every  we have
\end{proof}
\begin{Remark}
With some more effort in can be seen that in the above example we have that
\begin{multline*}
  \ov{\rm conv}(V(x))=\ov{\rm conv}
  \bigg(
  \big\{
    \frac1{13}\delta_{\ov{0}}+\frac3{13}\delta_{\ov{1}}+\frac9{13}
    \delta_{\ov{2}},\,\frac1{13}\delta_{\ov{2}}+\frac3{13}
    \delta_{\ov{0}}+\frac9{13}\delta_{\ov{1}},\\
    \frac1{13}\delta_{\ov{1}}+\frac3{13}\delta_{\ov{2}}+
    \frac9{13}\delta_{\ov{0}}
  \big\}
  \bigg),
\end{multline*}
and $V(x)$ is the combinatorial boundary of that simplex, that is,
\begin{multline*}
V(x)=\ov{\rm conv}
  \bigg(
  \big\{
    \frac1{13}\delta_{\ov{0}}+\frac3{13}
    \delta_{\ov{1}}+\frac9{13}\delta_{\ov{2}},\,
    \frac1{13}\delta_{\ov{1}}+\frac3{13}
    \delta_{\ov{2}}+\frac9{13}\delta_{\ov{0}}
  \big\}
  \bigg)\\ \cup  \ov{\rm conv}
  \bigg(
  \big\{
    \frac1{13}\delta_{\ov{2}}+\frac3{13}
    \delta_{\ov{0}}+\frac9{13}\delta_{\ov{1}},\,
    \frac1{13}\delta_{\ov{1}}+\frac3{13}
    \delta_{\ov{2}}+\frac9{13}\delta_{\ov{0}}
  \big\}
  \bigg)\\ \cup\ov{\rm conv}
  \bigg(
  \big\{
    \frac1{13}\delta_{\ov{0}}+\frac3{13}
    \delta_{\ov{1}}+\frac9{13}\delta_{\ov{2}},\,
    \frac1{13}\delta_{\ov{2}}+\frac3{13}\delta_{\ov{0}}+
    \frac9{13}\delta_{\ov{1}}
  \big\}
  \bigg).
\end{multline*}
\end{Remark}
%\section*{Acknowledgements}

\end{document}